\documentclass{article}
\usepackage{amsmath,amsthm,bm,colortbl,graphicx,mathrsfs,afterpage,amssymb,url}
\pdfoutput=1
\theoremstyle{plain}
\newtheorem{theorem}{Theorem}
\newtheorem{lemma}[theorem]{Lemma}

\newtheorem{corollary}[theorem]{Corollary}
\theoremstyle{remark}
\newtheorem*{remark}{Remark} 

\def\E{{\rm E}} 
\def\Var{{\rm Var}} 
\def\T{\textsf{T}} 
\def\MR{Mihailovi\'c and Rajkovi\'c}
\allowdisplaybreaks
\makeatletter 
\def\@biblabel#1{\hspace*{-\labelsep}}
\makeatother
\title{Parrondo games with spatial dependence}
\author{S. N. Ethier\\
\begin{small}University of Utah\end{small}\\
\begin{small}Department of Mathematics\end{small}\\
\begin{small}155 S. 1400 E., JWB 233\end{small}\\
\begin{small}Salt Lake City, UT 84112, USA\end{small}\\
\begin{small}e-mail: \url{ethier@math.utah.edu}\end{small}
\and
Jiyeon Lee\\
\begin{small}Yeungnam University\end{small}\\
\begin{small}Department of Statistics\end{small}\\
\begin{small}214-1 Daedong, Kyeongsan\end{small}\\
\begin{small}Kyeongbuk 712-749, South Korea\end{small}\\
\begin{small}e-mail: \url{leejy@yu.ac.kr}\end{small}}
\date{}

\begin{document}
\maketitle
\begin{abstract}
Toral introduced so-called cooperative Parrondo games, in which there are $N\ge3$ players arranged in a circle.  At each turn one player is randomly chosen to play.  He plays either game $A$ or game $B$.  Game $A$ results in a win or loss of one unit based on the toss of a fair coin.  Game $B$ results in a win or loss of one unit based on the toss of a biased coin, with the amount of the bias depending on whether none, one, or two of the player's two nearest neighbors have won their most recent games.  Game $A$ is fair, so the games are said to exhibit the Parrondo effect if game $B$ is losing or fair and the random mixture $(1/2)(A+B)$ is winning.  With the parameter space being the unit cube, we investigate the region in which the Parrondo effect appears.  Explicit formulas can be found if $3\le N\le 6$ and exact computations can be carried out if $7\le N\le 19$, at least.  We provide numerical evidence suggesting that the Parrondo region has nonzero volume in the limit as $N\to\infty$.
\end{abstract}

\section{Introduction}

Toral \cite{T} introduced what he called \textit{cooperative Parrondo games}, in which there are $N\ge3$ players labeled from 1 to $N$ and arranged in a circle in clockwise order.  At each turn, one player is chosen at random to play.  Call him player $i$.  He plays either game $A$ or game $B$, depending on the strategy.  In game $A$ he tosses a $p$-coin (i.e., $p$ is the probability of heads).  In game $B$, he tosses a $p_0$-coin if his neighbors $i-1$ and $i+1$ are both losers, a $p_1$-coin if $i-1$ is a loser and $i+1$ is a winner, a $p_2$-coin if $i-1$ is a winner and $i+1$ is a loser, and a $p_3$-coin if $i-1$ and $i+1$ are both winners.  (Because of the circular arrangement, player 0 is player $N$ and player $N+1$ is player 1.) A player's status as winner or loser depends on the result of his most recent game.  Of course, the player of either game wins one unit with heads and loses one unit with tails.  To initialize game $B$, we could assume that each player tosses a fair coin to establish his status as winner or loser; alternatively, we could specify an arbitrary initial distribution.  But because we are concerned with long-term behavior, the initial distribution is usually unimportant.  Under these assumptions, the model has an integer parameter $N$ and five probability parameters $p$, $p_0$, $p_1$, $p_2$, and $p_3$.

Toral used computer simulation to show that, with $N=50$, 100, or 200, $p=1/2$, $p_0=1$, $p_1=p_2=4/25$, and $p_3=7/10$, game $A$ is fair, game $B$ is losing, and the random mixture $C:=(1/2)(A+B)$ (toss a fair coin to determine which game to play) is winning, providing a new example of \textit{Parrondo's paradox} (Harmer and Abbott \cite{HA}, Abbott \cite{A}).  \MR\ \cite{MR} studied the case $N=3$ analytically and found that the Parrondo effect is not present for Toral's choice of the probability parameters but is present for other choices.  They also used analytical methods to examine the cases $4\le N\le 12$ assuming Toral's choice of the probability parameters.  Xie et al.\ \cite{Xetal} studied the case $N=4$ analytically but their game $A$ differs from Toral's, so their conclusions about the Parrondo effect do not apply here.  Our interest  is in the presence of the Parrondo effect for $N\ge3$ and arbitrary choices of the probability parameters.  We are willing to assume that $p=1/2$, so that game $A$ is fair, and that $p_1=p_2$ as others have assumed, so that the bias of the coin tossed in game $B$ depends only on the \textit{number} of winners among the two nearest neighbors.  We also assume that the random mixture of game $A$ and game $B$ is the equally weighted one, denoted above by $C$.  For fixed $N\ge3$, this still leaves three free probability parameters, $p_0$, $p_1$, and $p_3$, so our parameter space is the unit cube.

For $3\le N\le 6$, explicit formulas can be derived for $\mu_B$ and $\mu_C$, the mean profits per turn to the ensemble of $N$ players always playing game $B$ and always playing game $C$, respectively.  For $7\le N\le 19$ (and perhaps slightly larger $N$), an algorithm can be developed that generates exact values for $\mu_B$ and $\mu_C$, once the parameters $p_0$, $p_1$, and $p_3$ are specified.  By analyzing three choices of the parameter vector $(p_0,p_1,p_3)$ (including Toral's), we provide numerical evidence, but not a proof, that $\mu_B$ and $\mu_C$ converge as $N\to\infty$ and that the Parrondo effect (i.e., $\mu_B\le0$ and $\mu_C>0$) persists for all $N$ sufficiently large for a set of parameter vectors having nonzero volume.

Incidentally, there is also the concept of an anti-Parrondo effect (i.e., $\mu_B\ge0$ and $\mu_C<0$), and we show that there is a symmetry between the two concepts, so that the Parrondo region of the parameter space has the same volume as the anti-Parrondo region.  This should come as no surprise.  Indeed, as Harmer and Abbott \cite{HA} put it, ``In a practical sense this is like changing the observer's perspective of the games --- i.e.\ whether from the player's or the bank's point of view.''

In addition to the intrinsic appeal of creating a winning game from two fair or losing games, Parrondo's games have physical significance.  They were originally devised in 1996 by J. M. R. Parrondo as a pedagogical model of the flashing Brownian ratchet of Ajdari and Prost \cite{AD}.  Early work focussed on capital-dependent (Harmer and Abbott \cite{HA99}) and history-dependent (Parrondo, Harmer, and Abbott \cite{PHA}) games for a single player.  Multi-player games were introduced by Toral \cite{T,T2}, including not only the spatially dependent games studied here but also a model in which the role of the fair game $A$ is played by a forced transfer of one unit of wealth from one randomly chosen player to another, while game $B$ is the original capital- or history-dependent one.  This redistribution-of-wealth model was studied recently by the present authors \cite{EL12}.  That model also motivated a model of Xie et al. \cite{Xetal}, in which game $B$ is as in the spatially dependent Parrondo games as above, while game $A$ amounts to a forced transfer of one unit of wealth from one randomly chosen player to a randomly chosen nearest neighbor.  Other multi-player models include the synchronous spatial model of Mihailovi\'c and Rajkovi\'c \cite{MR2}; a two-dimensional spatial model by the same authors \cite{MR3}; a model of Amengual et al.\ \cite{AMCT} in which win probabilities depend on the number of winners (see also Arizmendi \cite{Ar}); and a model of Wang et al. \cite{Wetal} with dependence on capital parity.  As the literature on Parrondo's paradox approaches 200 papers, perhaps the easiest way to get an overview of the subject is to read the survey papers \cite{HA,A} cited above.

\section{The Markov chain and its reduction}\label{MC}

The Markov chain formalized by \MR\ \cite{MR} keeps track of the status (loser or winner, 0 or 1) of each of the $N\ge3$ players. Its state space is the product space
$$
\Sigma:=\{\bm x=(x_1,x_2,\ldots,x_N): x_i\in\{0,1\}{\rm\ for\ }i=1,\ldots,N\}=\{0,1\}^N
$$
with $2^N$ states.  With the help of some notation, we can specify the one-step transition matrix.  Let $m_i(\bm x):=2x_{i-1}+x_{i+1}$, or, in other words, $m_i(\bm x)$ is the integer (0, 1, 2, or 3) whose binary representation is $(x_{i-1}\,x_{i+1})_2$.  Of course $x_0:=x_N$ and $x_{N+1}:=x_1$.  Also, let $\bm x^i$ be the element of $\Sigma$ equal to $\bm x$ except at the $i$th component.  For example, $\bm x^1:=(1-x_1,x_2,x_3,\ldots,x_N)$.

The one-step transition matrix $\bm P$ for this Markov chain depends not only on $N$ but on four parameters, $p_0$, $p_1$, $p_2$, and $p_3$, which we assume satisfy $0<p_m<1$ for $m=0,1,2,3$.  (This rules out Toral's choice of the probability parameters, at least for now, but we will return to this point in Section~\ref{reducible}.  We do not assume that $p_1=p_2$ until Section~\ref{Parrondo}.)  It has the form
\begin{equation}\label{x to x^i}
P(\bm x,\bm x^i):=\begin{cases}N^{-1}p_{m_i(\bm x)}&\text{if $x_i=0$,}\\N^{-1}q_{m_i(\bm x)}&\text{if $x_i=1$,}\end{cases}\qquad i=1,\ldots,N,\;\bm x\in\Sigma,
\end{equation}
and
\begin{equation}\label{x to x}
P(\bm x,\bm x):=N^{-1}\bigg(\sum_{i:x_i=0}q_{m_i(\bm x)}+\sum_{i:x_i=1}p_{m_i(\bm x)}\bigg),\qquad \bm x\in\Sigma,
\end{equation}
where $q_m:=1-p_m$ for $m=0,1,2,3$ and empty sums are 0.  The Markov chain is irreducible and aperiodic.

The description of the model suggests that its long-term behavior should be invariant under rotation (and, if $p_1=p_2$, reflection) of the $N$ players.  In order to maximize the value of $N$ for which exact computations are feasible, we will use this idea to effectively reduce the size of the state space.  Our first result describes how this is done.  We omit the reasonably straightforward proof.

\begin{lemma}
Let $E$ be a finite set, fix $N\ge2$, let $G$ be a subgroup of the group of permutations of $(1,2,\ldots,N)$, and let $S$ be a subset of the product space $E^N:=E\times\cdots\times E$ with the property that $\bm x:=(x_1,\ldots,x_N)\in S$ implies $\bm x_\sigma:=(x_{\sigma(1)},\ldots,x_{\sigma(N)})\in S$ for all $\sigma\in G$.  Let $\bm P$ be the one-step transition matrix for an irreducible Markov chain in $S$, and let $\bm\pi$ be its unique stationary distribution.  Assume that $P(\bm x_\sigma,\bm y_\sigma)=P(\bm x,\bm y)$ for all $\sigma\in G$ and $\bm x,\bm y\in S$.  Then $\pi(\bm x_\sigma)=\pi(\bm x)$ for all $\sigma\in G$ and $\bm x\in S$.

Let us say that $\bm x\in S$ is equivalent to $\bm y\in S$ (written $\bm x\sim\bm y$) if there exists $\sigma\in G$ such that $\bm y=\bm x_\sigma$, and let us denote the equivalence class containing $\bm x$ by $[\bm x]$.  Then, in addition, $\bm P$ induces a one-step transition matrix $\bar{\bm P}$ for an irreducible Markov chain in the quotient set (i.e., the set of equivalence classes) $S/$$\sim$ defined by the formula
$$
\bar{P}([\bm x],[\bm y]):=\sum_{\bm y':\bm y'\sim\bm y}P(\bm x,\bm y')=\sum_{\sigma\in G:\bm y_\sigma\;{\rm distinct}}P(\bm x,\bm y_\sigma),
$$
the second sum extending over only those $\sigma\in G$ for which the various $\bm y_\sigma$ are distinct.  Furthermore, if $\bar{\bm\pi}$ is the unique stationary distribution for $\bar{\bm P}$, then $\bm\pi$ is given by $\pi(\bm x)=\bar{\pi}([\bm x])/|[\bm x]|$, where $|[\bm x]|$ denotes the cardinality of the equivalence class $[\bm x]$.
\end{lemma}

\begin{remark}
The case $S=E^N$ is of primary interest, but examples in which $S$ is a proper subset of $E^N$ appear in Section~\ref{reducible}.

\end{remark}

The lemma applies to our Markov chain if $G$ is the subgroup of \textit{cyclic permutations} (or rotations) of $(1,2,\ldots,N)$, that is, the group generated by
\begin{equation}\label{cyclic}
(\sigma(1),\sigma(2),\ldots,\sigma(N)):=(2,3,\ldots,N,1).
\end{equation}
Indeed, for any cyclic permutation $\sigma$,
\begin{eqnarray}\label{invariance_property}
P(\bm x_\sigma,(\bm x^i)_\sigma)=P(\bm x_\sigma,(\bm x_\sigma)^{\sigma^{-1}(i)})&=&\begin{cases}N^{-1}p_{m_{\sigma^{-1}(i)}(\bm x_\sigma)}&\text{if $(\bm x_\sigma)_{\sigma^{-1}(i)}=0$}\\
N^{-1}q_{m_{\sigma^{-1}(i)}(\bm x_\sigma)}&\text{if $(\bm x_\sigma)_{\sigma^{-1}(i)}=1$}\end{cases}\nonumber\\
&=&\begin{cases}N^{-1}p_{m_i(\bm x)}&\text{if $x_i=0$}\\N^{-1}q_{m_i(\bm x)}&\text{if $x_i=1$}\end{cases}\nonumber\\
&=&P(\bm x,\bm x^i)
\end{eqnarray}
for $i=1,\ldots,N$ and all $\bm x\in\Sigma$.  If $p_1=p_2$, then (\ref{invariance_property}) also applies to the \textit{order-reversing permutation} (or reflection) of $(1,2,\ldots,N)$,
\begin{equation}\label{reverse}
(\sigma(1),\sigma(2),\ldots,\sigma(N)):=(N,N-1,\ldots,2,1);
\end{equation}
note that the third equality uses $p_1=p_2$.  In this case $G$ is the group generated by (\ref{cyclic}) and (\ref{reverse}) and is known as the \textit{dihedral group} of order $2N$.

The practical effect of this is that we can reduce the size of the state space (namely, $2^N$) to what we will call its \textit{effective size}, which is simply the number of equivalence classes.  For example, if $N=3$, there are eight states and four equivalence classes, namely
$$
0=\{000\},\quad
1=\{001,010,100\},\quad
2=\{011,101,110\},\quad
3=\{111\}.
$$
Notice that we label equivalence classes by the number of 1s each element has.
If $N=4$, there are 16 states and six equivalence classes, namely
\begin{eqnarray*}
0&=&\{0000\},\\
1&=&\{0001,0010,0100,1000\},\\
2&=&\{0011,0110,1001,1100\},\\
2'&=&\{0101,1010\},\\
3&=&\{0111,1011,1101,1110\},\\
4&=&\{1111\}.
\end{eqnarray*}
If $N=6$, there are 64 states and 14 equivalence classes.  A more concise notation in this case is
$$
[0]_1/[1]_6/[3]_6,[5]_6,[9]_3/[7]_6,[11]_6,[13]_6,[21]_2/[15]_6,[23]_6,[27]_3/[31]_6/[63]_1.
$$
Here each equivalence class is described by its least element in decimal form.  Subscripts indicate equivalence class sizes.  Equivalence classes are ordered first by the number of 1s each element has in binary form (the $/$ symbol separates different numbers of 1s), and second by the least element.  If in addition $p_1=p_2$, then, using the order-reversing permutation, equivalence classes $[11]_6$ and $[13]_6$ coalesce and are replaced by $[11]_{12}$, so there are only 13 equivalence classes in that case.

The number of equivalence classes with $G$ being the group of cyclic permutations follows the sequence A000031 in the \textit{The On-Line Encyclopedia of Integer Sequences} (\url{http://oeis.org/}), described as the number of necklaces with $N$ beads of two colors when turning over is not allowed.  There is an explicit formula in terms of Euler's phi-function.  If $p_1=p_2$, we can reverse the order of the players, and the number of equivalence classes with $G$ being the dihedral group follows the sequence A000029 in the \textit{OEIS}, described as the number of necklaces with $N$ beads of two colors when turning over is allowed.  Again there is an explicit formula.  See Table~\ref{sequences}.

\begin{table}[ht]
\caption{\label{sequences}The size and effective size of the state space when there are $N$ players.\medskip}
\catcode`@=\active \def@{\hphantom{0}}
\begin{center}
\begin{footnotesize}
\begin{tabular}{cccc}
\noalign{\smallskip}
\hline
\noalign{\smallskip}
number of   &  size of   &  effective size  &  effective size \\
players & @state space@ & not assuming  & assuming  \\
 $N$ & $2^N$ & $p_1=p_2$ & $p_1=p_2$ \\
\noalign{\smallskip}
\hline
\noalign{\smallskip}
@3 & @@@@@@8 & @@@@4 & @@@@4 \\
@4 & @@@@@16 & @@@@6 & @@@@6 \\
@5 & @@@@@32 & @@@@8 & @@@@8 \\
@6 & @@@@@64 & @@@14 & @@@13 \\
@7 & @@@@128 & @@@20 & @@@18 \\
@8 & @@@@256 & @@@36 & @@@30 \\
@9 & @@@@512 & @@@60 & @@@46 \\
10 & @@@1024 & @@108 & @@@78 \\
11 & @@@2048 & @@188 & @@126 \\
12 & @@@4096 & @@352 & @@224 \\
13 & @@@8192 & @@632 & @@380 \\
14 & @@16384 & @1182 & @@687 \\
15 & @@32768 & @2192 & @1224 \\
16 & @@65536 & @4116 & @2250 \\
17 & @131072 & @7712 & @4112 \\
18 & @262144 & 14602 & @7685 \\
19 & @524288 & 27596 & 14310 \\
20 & 1048576 & 52488 & 27012 \\
\noalign{\smallskip}
\hline
\end{tabular}
\end{footnotesize}
\end{center}
\end{table}

Elements of $\Sigma$ are most naturally ordered by regarding them as the binary representations of the integers $0,1,\ldots,2^N-1$.  Elements of $\Sigma/$$\sim$ have two natural orderings, one of which was described above in connection with the case $N=6$.  Another approach, which is computationally simpler (and adopted in the Appendix), is to order equivalence classes simply by the least element.

To illustrate the one-step transition matrix on the reduced state space, consider the case $N=3$.  The $8\times8$ one-step transition matrix $\bm P$ is equal to the transpose of (9) in \MR\ \cite{MR}, whereas
\begin{equation}\label{Pbar_N=3}
\setlength{\arraycolsep}{1mm}
\bar{\bm P}={1\over3}\left(\begin{array}{cccc}
3 q_0 & 3 p_0 & 0 & 0 \\
q_0 & p_0 + q_1 + q_2 & p_1 + p_2 & 0 \\
0 & q_1 + q_2 & p_1 + p_2 + q_3 & p_3 \\
0 & 0 & 3 q_3 & 3 p_3
\end{array}\right).
\end{equation}
Consider also the case $N=4$.  The $16\times16$ one-step transition matrix $\bm P$ is equal to (12) in Xie et al.\ \cite{Xetal}, whereas
\begin{equation}\label{Pbar_N=4}
\setlength{\arraycolsep}{1mm}
\bar{\bm P}={1\over4}\left(\begin{array}{cccccc}
4 q_0 & 4 p_0 & 0 & 0 & 0 & 0 \\
q_0 & 1 + q_1 + q_2 & p_1 + p_2 & p_0 & 0 & 0 \\
0 & q_1 + q_2 & 2 & 0 & p_1 + p_2 & 0 \\
0 & 2 q_0 & 0 & 2 (p_0 + q_3) & 2 p_3 & 0 \\
0 & 0 & q_1 + q_2 & q_3 & 1 + p_1 + p_2 & p_3 \\
0 & 0 & 0 & 0 & 4 q_3 & 4 p_3
\end{array}\right),
\end{equation}
where rows and columns are labeled by $0,1,2,2',3,4$, the two suggested methods being equivalent.

More generally, we can give a fairly explicit formula for $\bar{\bm P}$.  First, define the function $s:\Sigma/$$\sim$${}\mapsto\{0,1,\ldots,N\}$ by $s([\bm x]):=x_1+x_2+\cdots+x_N$; it counts the number of 1s in each element of an equivalence class and is clearly well defined.  Then
\begin{equation}\label{Pbar}
\bar{P}([\bm x],[\bm y])=\begin{cases}N^{-1}\big(\sum_{i:x_i=0}q_{m_i(\bm x)}+\sum_{i:x_i=1}p_{m_i(\bm x)}\big)&\text{if $[\bm y]=[\bm x]$}\\
N^{-1}\sum_{i:x_i=1,\bm x^i\sim\bm y}q_{m_i(\bm x)}&\text{if $s([\bm y])=s([\bm x])-1$}\\
N^{-1}\sum_{i:x_i=0,\bm x^i\sim\bm y}p_{m_i(\bm x)}&\text{if $s([\bm y])=s([\bm x])+1$}\\
0&\text{otherwise}\end{cases}
\end{equation}
for all $[\bm x],[\bm y]\in \Sigma/$$\sim$.  The first case in (\ref{Pbar}) follows because $\bm x^i\not\sim \bm x$ for all $i$.  Note also that, even if $s([\bm y])=s([\bm x])\pm1$, we may have $\bar{P}([\bm x],[\bm y])=0$; for example, $\bar{P}([001001],[010101])=0$.  We say ``fairly explicit'' because the evaluation of $\bar{\bm P}$ still requires an enumeration of $\Sigma/$$\sim$, which can be time consuming.

\section{The stationary distribution}

The unique stationary distribution $\bm\pi$ is too complicated to expect explicit formulas except for $3\le N\le 6$.  For $N=3$, we can find the unique stationary distribution for our eight-state chain by doing the same for a four-state chain.  An invariant measure for the four-state chain is $(\rho_0,3\rho_1,3\rho_2,\rho_3)$, where
$$
\rho_0:=q_0(q_1+q_2)q_3,\;
\rho_1:=p_0(q_1+q_2)q_3,\;
\rho_2:=p_0(p_1+p_2)q_3,\;
\rho_3:=p_0(p_1+p_2)p_3,
$$
so an invariant measure for the eight-state chain is $(\rho_0,\rho_1,\rho_1,\rho_2,\rho_1,\rho_2,\rho_2,\rho_3)$.  The stationary distribution follows by dividing each entry by the sum $\rho_0+3\rho_1+3\rho_2+\rho_3$.  Incidentally, the $N=3$ case is the only case for which the stationary distribution is reversible ($\bar{\bm\pi}$ in general, and $\bm\pi$ if $p_1=p_2$).  A partial explanation is that (\ref{Pbar_N=3}) is tridiagonal, that is, it corresponds to a birth-and-death chain.

For $N=4$ we can find the unique stationary distribution for our 16-state chain by doing the same for a six-state chain.  An invariant measure for the six-state chain is $(\rho_0,4\rho_1,4\rho_2,2\rho_{2'},4\rho_3,\rho_4)$, where
\begin{eqnarray*}
\rho_0&:=&q_0[2q_0q_3+(q_1+q_2)^2(q_0+p_3)]q_3,\\
\rho_1&:=&p_0[2q_0q_3+(q_1+q_2)^2(q_0+p_3)]q_3,\\
\rho_2&:=&p_0[2q_0q_3+(p_1+p_2)(q_1+q_2)(q_0+p_3)+(q_1+q_2)(p_3-q_0)]q_3\\
&\;=&p_0[2p_0p_3+(p_1+p_2)(q_1+q_2)(q_0+p_3)+(p_1+p_2)(q_0-p_3)]q_3,\\
\rho_{2'}&:=&p_0[2p_0q_3+(p_1+p_2)^2q_3+(q_1+q_2)^2p_0]q_3,\\
\rho_3&:=&p_0[2p_0p_3+(p_1+p_2)^2(q_0+p_3)]q_3,\\
\rho_4&:=&p_0[2p_0p_3+(p_1+p_2)^2(q_0+p_3)]p_3;
\end{eqnarray*}
here we have given two formulas for $\rho_2$ so that we can see at a glance that it is positive.  Thus, an invariant measure for the 16-state chain is $(\rho_0,\rho_1,\rho_1,\rho_2,\rho_1,\rho_{2'},\break\rho_2,\rho_3,\rho_1,\rho_2,\rho_{2'},\rho_3,\rho_2,\rho_3,\rho_3,\rho_4)$.
The stationary distribution follows by dividing each entry by the sum $\rho_0+4\rho_1+4\rho_2+2\rho_{2'}+4\rho_3+\rho_4$.

In a similar way, we have also found formulas for the unique stationary distribution in the cases $N=5$ and $N=6$ (assuming $p_1=p_2$ in the latter case), but they are considerably more complicated and consequently will not be given here.  In particular, we have not shown algebraically, as we have for $N=3$ and $N=4$, that each term is positive; for that we must rely on Markov chain theory.

Notice that, when $N=3$ or $N=4$, the stationary distribution depends on $p_1$ and $p_2$ only through $p_1+p_2$.  This property also holds when $N=5$ but fails when $N\ge6$.  When it holds, it implies that there is no loss of generality in assuming $p_1=p_2$.  The reason this property holds when $N=3$ or $N=4$ is that $\bar{\bm P}$ in (\ref{Pbar_N=3}) and (\ref{Pbar_N=4}) depends on $p_1$ and $p_2$ only through $p_1+p_2$.  The same is true when $N=5$ but not when $N=6$ because, for example, $\bar{P}([001011],[011011])=(1/6)p_1$.

\section{Strong law of large numbers}

Is there a strong law of large numbers (SLLN) and a central limit theorem for the sequence of profits by the ensemble of $N$ players?  More specifically, does Theorem~1 of Ethier and Lee \cite{EL} apply in this context?  Let us recall the statement of that theorem.

Consider an irreducible aperiodic Markov chain $\{X_n\}_{n\ge0}$ with finite state space $\Sigma_0$.  It evolves according to the one-step transition matrix ${\bm P}=(P_{ij})_{i,j\in\Sigma_0}$.  Let us denote its unique stationary distribution by the row vector ${\bm \pi}=(\pi_i)_{i\in \Sigma_0}$.  Let $w:\Sigma_0\times\Sigma_0\mapsto {\bf R}$ be an arbitrary function, which we write as a matrix ${\bm W}=(w(i,j))_{i,j\in\Sigma_0}$ and refer to as the \textit{payoff matrix}. Define the sequences $\{\xi_n\}_{n\ge1}$ and $\{S_n\}_{n\ge1}$ by
\begin{equation}\label{xi_n}
\xi_n:=w(X_{n-1},X_n),\qquad n\ge1,
\end{equation}
and
\begin{equation}\label{S_n}
S_n:=\xi_1+\cdots+\xi_n,\qquad n\ge1.
\end{equation}
Let ${\bm \Pi}$ denote the square matrix each of whose rows is ${\bm \pi}$, and let ${\bm Z}:=({\bm I}-({\bm P}-{\bm \Pi}))^{-1}$ denote the \textit{fundamental matrix}.  Denote by $\dot{\bm P}$ and $\ddot{\bm P}$ the Hadamard (entrywise) products $\bm P\circ\bm W$ and $\bm P\circ\bm W\circ\bm W$ (so $\dot{P}_{ij}:=P_{ij}w(i,j)$ and $\ddot{P}_{ij}:=P_{ij}w(i,j)^2$).  Let $\bm 1:=(1,1,\ldots,1)^\T$ and define
\begin{equation}\label{mu,sigma2}
\mu:=\bm\pi\dot{\bm P}\bm 1\quad{\rm and}\quad\sigma^2:=\bm\pi\ddot{\bm P}\bm 1
-(\bm\pi\dot{\bm P}\bm 1)^2+2\bm\pi\dot{\bm P}(\bm Z-\bm\Pi)\dot{\bm P}\bm 1.
\end{equation}

\begin{theorem}[Ethier and Lee \cite{EL}]
Under the above assumptions, and with the distribution of $X_0$ arbitrary, $\lim_{n\to\infty}n^{-1}\E[S_n]=\mu$,
$$
{S_n\over n}\to \mu\;\;{\rm a.s.},
$$
$\lim_{n\to\infty}n^{-1}\Var(S_n)=\sigma^2$, and, if $\sigma^2>0$,
$$
{S_n-n\mu\over\sqrt{n\sigma^2}}\to_d N(0,1).
$$
If $\mu=0$ and $\sigma^2>0$, then $-\infty=\liminf_{n\to\infty}S_n<\limsup_{n\to\infty}S_n=\infty$ \emph{a.s.}
\end{theorem}

\begin{remark}
The abbreviation ``a.s.'' stands for ``almost surely,'' meaning ``with probability 1.''  The symbol $\to_d$ denotes convergence in distribution.
A game is \textit{winning} if $\mu>0$ (hence $\lim S_n=\infty$ a.s.), \textit{losing} if $\mu<0$ (hence $\lim S_n=-\infty$ a.s.), and \textit{fair} if $\mu=0$ (hence $-\infty=\liminf S_n<\limsup S_n=\infty$ a.s., assuming $\sigma^2>0$).
\end{remark}

It appears at first glance that the theorem does not apply in the present context because the payoffs are not completely specified by the one-step transitions of the Markov chain.  Specifically, a transition from  a state $\bm x$ to itself results whenever a loser loses or a winner wins, and the transition probability is given by (\ref{x to x}).  And yet (\ref{xi_n}) and (\ref{S_n}) suggest that $w(\bm x,\bm x)$ should be $\pm1$.  Does the theorem need to be generalized so that $w(i,j)$ is not simply the payoff when $X_{n-1}=i$ and $X_n=j$ but rather is the conditional expected payoff given $X_{n-1}=i$ and $X_n=j$?  Actually, it is more convenient to leave the theorem as it is and instead generalize the Markov chain.

Our original Markov chain has state space $\Sigma:=\{0,1\}^N$ and its one-step transition matrix $\bm P$ is given by (\ref{x to x^i}) and (\ref{x to x}).  Let $\bm\pi$ denote its unique stationary distribution.  We augment the state space, letting $\Sigma^*:=\Sigma\times \{1,2,\ldots,N\}$ and keeping track not only of the status of each player as described by $\bm x\in\Sigma$ but also of the label of the next player to play, say $i$.  The new one-step transition matrix $\bm P^*$ has the form
$$
P^*((\bm x,i),(\bm x^i,j)):=\begin{cases}N^{-1}p_{m_i(\bm x)}&\text{if $x_i=0$,}\\N^{-1}q_{m_i(\bm x)}&\text{if $x_i=1$,}\end{cases}\qquad (\bm x,i)\in\Sigma^*,\; j=1,2,\ldots,N,
$$
and
$$
P^*((\bm x,i),(\bm x,j)):=\begin{cases}N^{-1}q_{m_i(\bm x)}&\text{if $x_i=0$,}\\N^{-1}p_{m_i(\bm x)}&\text{if $x_i=1$,}\end{cases}\qquad (\bm x,i)\in\Sigma^*,\; j=1,2,\ldots,N,
$$
where $q_m:=1-p_m$ for $m=0,1,2,3$.  This remains an irreducible aperiodic Markov chain, and its unique stationary distribution $\bm\pi^*$ is given by $\pi^*(\bm x,i)=N^{-1}\pi(\bm x)$.  Further, the payoff matrix now has each nonzero entry equal to $\pm1$, so the theorem applies.

Therefore, by (\ref{mu,sigma2}), the mean parameter in the SLLN has the form
$$
\mu=\bm\pi^*\dot{\bm P}^*\bm1=\sum_{\bm x\in\Sigma}\pi(\bm x)\sum_{i=1}^N N^{-1}[p_{m_i(\bm x)}-q_{m_i(\bm x)}].
$$
Alternatively, this can be rewritten as
\begin{equation}\label{rule}
\mu=\bm\pi\dot{\bm P}\bm1=\bar{\bm\pi}\dot{\bar{\bm P}}{\bm1},
\end{equation}
where $\bm1$ is the column vector of 1s of the appropriate dimension and $\dot{\bm P}$ and $\dot{\bar{\bm P}}$ have new meanings.  Specifically, $\dot{\bm P}$ is obtained from $\bm P$, and $\dot{\bar{\bm P}}$ from $\bar{\bm P}$, by replacing each $q_m$ by $-q_m$.  This ``rule of thumb'' requires some caution:  It must be applied before any simplifications are made using $q_m=1-p_m$.  For example, the rule applies directly to (\ref{Pbar_N=3}) but not to (\ref{Pbar_N=4}) because the 1s in the $(1,1)$ and $(3,3)$ entries are $p_0+q_0$ and $p_3+q_3$, respectively, and the 2 in the $(2,2)$ entry is $p_1+q_1+p_2+q_2$.  (Recall that rows and columns are labeled by $0,1,2,2',3,4$.)

The main consequence of Theorem~1 from our perspective is the SLLN it ensures.  Let $S_n^A$, $S_n^B$, and $S_n^C$ denote the cumulative profit after $n$ turns to the ensemble of $N$ players playing game $A$, $B$, and $C:=\gamma A+(1-\gamma)B$.  Then
$$
{S_n^A\over n}\to \mu_A=2p-1\;\;{\rm a.s.},\qquad{S_n^B\over n}\to \mu_B\;\;{\rm a.s.},\qquad{S_n^C\over n}\to \mu_C\;\;{\rm a.s.}
$$
Of course, there is also a central limit theorem, but we will not try to evaluate its variance parameter.

We conclude this section with an application of the SLLN. Let us denote $\mu$ of (\ref{rule}) by $\mu(p_0,p_1,p_2,p_3)$ to emphasize its dependence on the probability parameters. (We do not assume that $p_1=p_2$.)

\begin{corollary}
With $q_m:=1-p_m$ for $m=0,1,2,3$, we have
$$
\mu(p_0,p_1,p_2,p_3)=-\mu(q_3,q_2,q_1,q_0).
$$
In particular, if $p_0+p_3=1$ and $p_1+p_2=1$, then $\mu(p_0,p_1,p_2,p_3)=0$.
\end{corollary}

\begin{proof}
We prove this via a coupling argument.  Define the function $\eta:\Sigma\mapsto\Sigma$ by $\eta(\bm x):=(1-x_1,1-x_2,\ldots,1-x_N)$.  Let $\{\bm X(n)\}$ be the Markov chain in $\Sigma$ with initial state $\bm x$ and one-step transition matrix $\bm P^{p_0,p_1,p_2,p_3}$ given by (\ref{x to x^i}) and (\ref{x to x}), where the superscripts are merely intended to emphasize the probability parameters.  Then $\bm X'(n):=\eta(\bm X(n))$ defines a Markov chain $\{\bm X'(n)\}$ in $\Sigma$ with initial state $\bm x':=\eta(\bm x)$ and one-step transition matrix $\bm P^{q_3,q_2,q_1,q_0}$.  The two processes are coupled, using the same sequence of players and the same sequence of coin tosses.

To help clarify this, let us consider the special case $N=3$, in which there are three players 1, 2, and 3 whose collective status is described by $\{\bm X(n)\}$ and three players $1'$, $2'$, and $3'$ whose collective status is described by $\{\bm X'(n)\}$.  Suppose $\bm X(0)=\bm x:=(0,1,0)$ and $\bm X'(0)=\bm x':=\eta(\bm x)=(1,0,1)$.  If player 2 is chosen to play (and determine $\bm X(1)$), then player $2'$ will also be chosen to play (and determine $\bm X'(1)$).  Player 2 is required to toss a $p_0$-coin because $0=(0\,0)_2$ and the first probability parameter of $\{\bm X(n)\}$ is $p_0$, and player $2'$ is required to toss a $q_0$-coin because $3=(1\,1)_2$ and the fourth probability parameter of $\{\bm X'(n)\}$ is $q_0$.  We can use the same coin toss for both players, except that if player 2 sees heads, then player $2'$ sees tails, and vice versa.  So a loss by player 2 (equivalently, a win by player $2'$) results in $\bm X(1)=(0,0,0)$ and $\bm X'(1)=(1,1,1)$, whereas a win by player 2 (equivalently, a loss by player $2'$) results in $\bm X(1)=\bm x$ and $\bm X'(1)=\bm x'$.  The coupling proceeds in this manner at each turn.

Let $S_n^{p_0,p_1,p_2,p_3}$ be the cumulative profit to the ensemble of $N$ players after $n$ turns, where again the superscripts emphasize the probability parameters.  Then, by the SLLN,
\begin{eqnarray*}
\mu(p_0,p_1,p_2,p_3)&=&\lim_{n\to\infty}n^{-1}S_n^{p_0,p_1,p_2,p_3}\\
&=&-\lim_{n\to\infty}n^{-1}S_n^{q_3,q_2,q_1,q_0}=-\mu(q_3,q_2,q_1,q_0)\;\;\text{a.s.},
\end{eqnarray*}
where the middle equality holds because each win in the $\{\bm X(n)\}$ process corresponds to a loss in the $\{\bm X'(n)\}$ process and vice versa.
\end{proof}

\section{Reducible cases}\label{reducible}

We have assumed that $0<p_m<1$ for $m=0,1,2,3$, which ensures that our Markov chain is irreducible and aperiodic.  Can we weaken this assumption?  Let us consider several cases.  We denote by $\bm0\in\Sigma$ the state consisting of all 0s, and by $\bm1\in\Sigma$ the state consisting of all 1s.\medskip

1. Suppose $p_0=1$ and $0<p_m<1$ for $m=1,2,3$, as Toral \cite{T} originally assumed.  Then state $\bm 0$ cannot be reached from $\Sigma-\{\bm 0\}$ and $\bm P$, with row $\bm 0$ and column $\bm 0$ deleted, is a stochastic matrix that is irreducible and aperiodic.  Note that Lemma~1 is applicable with $S:=\Sigma-\{\bm 0\}$.  Theorem~1 also applies with $\Sigma_0:=S$, or we could take $\Sigma_0:=\Sigma$; it does not matter because $\bm P$ will have a unique stationary distribution $\bm\pi$, which necessarily satisfies $\pi(\bm0)=0$.  The reason we might prefer $\Sigma_0:=\Sigma$ is that any formula obtained assuming $0<p_m<1$ for $m=0,1,2,3$ will remain valid after substituting $p_0=1$.\medskip

2. Suppose $p_0=0$ and $0<p_m<1$ for $m=1,2,3$.  Then state $\bm 0$ is absorbing, and absorption eventually occurs with probability 1.  Hence $S_n^B-S_{n-1}^B=-1$ for all $n$ sufficiently large, so $\mu_B=-1$, as noticed by Xie et al.\ \cite{Xetal}.\medskip

3. Suppose $p_3=0$ and $0<p_m<1$ for $m=0,1,2$.  This is analogous to case 1.\medskip

4. Suppose $p_3=1$ and $0<p_m<1$ for $m=0,1,2$.  This is analogous to case 2, except $\mu_B=1$.\medskip

5. Suppose $p_0=1$, $p_3=0$, and $0<p_m<1$ for $m=1,2$.  Then states $\bm 0$ and $\bm 1$ cannot be reached from $\Sigma-\{\bm 0,\bm 1\}$ and $\bm P$, with rows $\bm 0$ and $\bm 1$ and columns $\bm 0$ and $\bm 1$ deleted, is a stochastic matrix.  However, it is not irreducible (unless $N=3$), so Lemma~1 does not apply with $S:=\Sigma-\{\bm0,\bm1\}$.  If $N$ is even, then the two states $0101\cdots01$ and $1010\cdots10$ in which 0s and 1s alternate are absorbing, and from either state there is a win of one unit with probability 1/2 and a loss of one unit with probability 1/2.  Consequently, $\mu_B=0$, regardless of $p_1$ and $p_2$.\medskip

6. Suppose $p_0=0$, $p_3=1$, and $0<p_m<1$ for $m=1,2$.  Then both $\bm 0$ and $\bm 1$ are absorbing, and absorption occurs with probability 1.  The probability of absorption at $\bm 1$ depends on the initial state (or equivalence class), and can be calculated for small $N\ge3$.   For example, in the case $N=4$ with $p_1=p_2$, we can derive formulas for $\mu_B$ as a function of $p_1$, depending on the equivalence class of the initial state, and the results are consistent with the simulations of Xie et al.\ \cite{Xetal};  compare their Fig.\ II.1 (p.\ 410).  The details are left to the reader.

\section{The Parrondo region}\label{Parrondo}

If we denote $\mu$ of (\ref{rule}) by $\mu(p_0,p_1,p_2,p_3)$ to emphasize its dependence on the probability parameters, then the mean profits per turn for game $A$, game $B$, and game $C:=\gamma A+(1-\gamma)B$ are $\mu_A:=\mu(p,p,p,p)=2p-1$, $\mu_B:=\mu(p_0,p_1,p_2,p_3)$, and $\mu_C:=\mu(r_0,r_1,r_2,r_3)$, where $r_m:=\gamma p+(1-\gamma)p_m$ for $m=0,1,2,3$.  The \textit{Parrondo effect} is said to be present if $\mu_A\le0$, $\mu_B\le0$, and $\mu_C>0$.  The \textit{anti-Parrondo effect} is said to be present if $\mu_A\ge0$, $\mu_B\ge0$, and $\mu_C<0$.

In what follows we assume for convenience that $p=1/2$ (game $A$ is fair), $p_1=p_2$ (the bias of the coin tossed in game $B$ depends only on the \textit{number} of winners among the two nearest neighbors), and $\gamma=1/2$ (the random mixture of games $A$ and $B$ is the equally weighted one).  In particular, for fixed $N\ge3$, we have three free probability parameters, $p_0$, $p_1$, and $p_3$, so our parameter space is the unit cube $(0,1)^3:=(0,1)\times(0,1)\times(0,1)$.  With caution (see Section~\ref{reducible}), we can also include parts of the boundary.  Of interest are the regions in the parameter space in which the Parrondo effect (i.e., $\mu_B\le0$ and $\mu_C>0$) appears and the anti-Parrondo effect (i.e., $\mu_B\ge0$ and $\mu_C<0$) appears;  let us refer to them as the \textit{Parrondo region} and \textit{anti-Parrondo region}.

\begin{theorem}
Fix $N\ge3$ and assume as above that $p=1/2$ in game $A$, $p_1=p_2$ in games $B$ and $C$, and $\gamma=1/2$ in game $C$.  With $q_m:=1-p_m$ for $m=0,1,3$, the parameter vector $(p_0,p_1,p_3)$ belongs to the Parrondo region if and only if the parameter vector $(q_3,q_1,q_0)$ belongs to the anti-Parrondo region.  In particular, the Parrondo region and the anti-Parrondo region have the same volume.
\end{theorem}

\begin{proof}
Let $\mu_B$ and $\mu_C$ denote the means for the parameter vector $(p_0,p_1,p_1,p_3)$, and let $\mu_B^*$ and $\mu_C^*$ denote the means for the parameter vector $(q_3,q_1,q_1,q_0)$.  Then, by Corollary~1 (and using $p_1=p_2$),
$\mu_B=\mu(p_0,p_1,p_1,p_3)=-\mu(q_3,q_1,q_1,q_0)=-\mu_B^*$ and
$\mu_C=\mu((1/2+p_0)/2,(1/2+p_1)/2,(1/2+p_1)/2,(1/2+p_3)/2)
=-\mu((1/2+q_3)/2,(1/2+q_1)/2,(1/2+q_1)/2,(1/2+q_0)/2)=-\mu_C^*$.
Therefore, $\mu_B\le0$ and $\mu_C>0$ if and only if $\mu_B^*\ge0$ and $\mu_C^*<0$.

For the second conclusion, we define the mapping $\Lambda:(0,1)^3\mapsto(0,1)^3$ by
$\Lambda(p_0,p_1,\break p_3)=(1-p_3,1-p_1,1-p_0)$.
This one-to-one transformation has Jacobian identically equal to 1, so it is measure preserving.  Since it maps the Parrondo region onto the anti-Parrondo region, the two regions must have the same volume.
\end{proof}

It will therefore suffice to focus our attention in what follows on the Parrondo region.

\subsection{$N=3$}

Using the stationary distribution derived above and the mean formula (\ref{rule}), together with the assumption that $p_1=p_2$, we find that
\begin{equation}\label{meanB_N=3,p1=p2}
\mu_B={p_1 (p_0 + q_3) - q_3\over p_0 p_1 + 2 p_0 q_3 + q_1 q_3}.
\end{equation}
Since $p=1/2$ in game $A$ and $\gamma=1/2$ in game $C$, by (\ref{meanB_N=3,p1=p2}) with $p_m$ replaced by $(1/2+p_m)/2$ for $m=0,1,3$, we have
$$
\mu_C={2 p_1 (1 + p_0 + q_3) - q_0 - 3 q_3\over 1 + 3p_0 + 2p_0 p_1 + 4p_0 q_3 + 2 p_1 p_3 + 2 q_1 + 5 q_3}.
$$
Therefore, the Parrondo region is described by
$p_1 (p_0 + q_3) - q_3\le 0$ and $2 p_1 (1 + p_0 + q_3) - q_0 - 3 q_3>0$
or, equivalently,
\begin{equation}\label{Parrondo_N=3}
{q_0+3q_3\over2(1+p_0+q_3)}<p_1\le{q_3\over p_0+q_3}.
\end{equation}
The first inequality is equivalent to (24) in \MR\ \cite{MR}.
There exists such a $p_1$ if and only if $\min(p_0,q_0)<p_3<\max(p_0,q_0)$.  In particular, the area of the region in the $(p_0,p_3)$ unit square for which there exists a $p_1$ satisfying (\ref{Parrondo_N=3}) is equal to 1/2.

With the parameter space being the $(p_0,p_3,p_1)$ unit cube, the Parrondo region is the union of two connected components.  See Figure~\ref{fig}.  Two straightforward iterated integrals yield its exact volume, $(9\ln9-8\ln8-3)/8\approx0.0174361$.

\begin{figure}[ht]
\centering
\includegraphics[width = 2.2in]{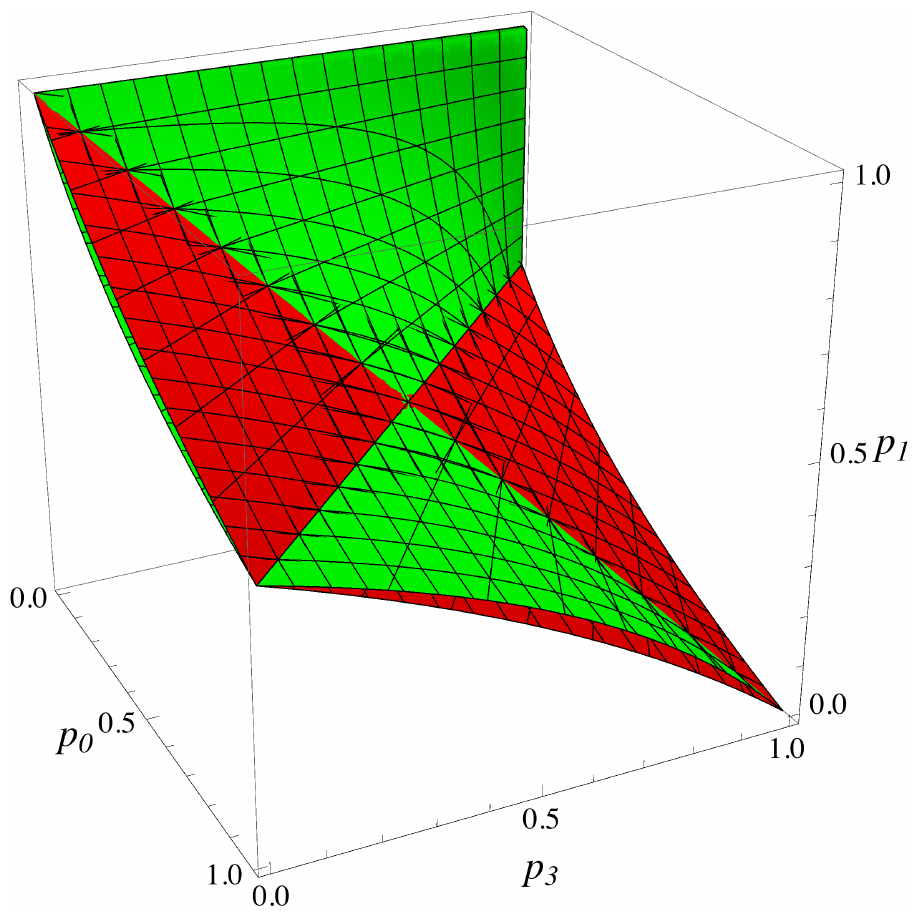}
\includegraphics[width = 2.2in]{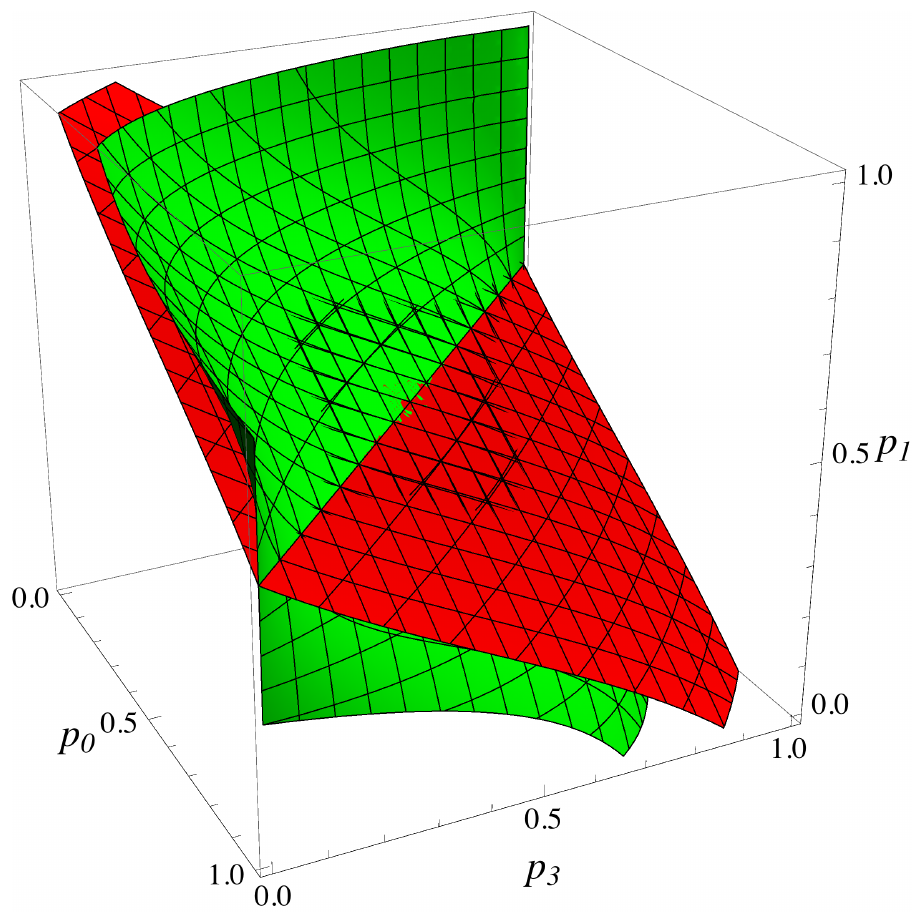}\par
\includegraphics[width = 2.2in]{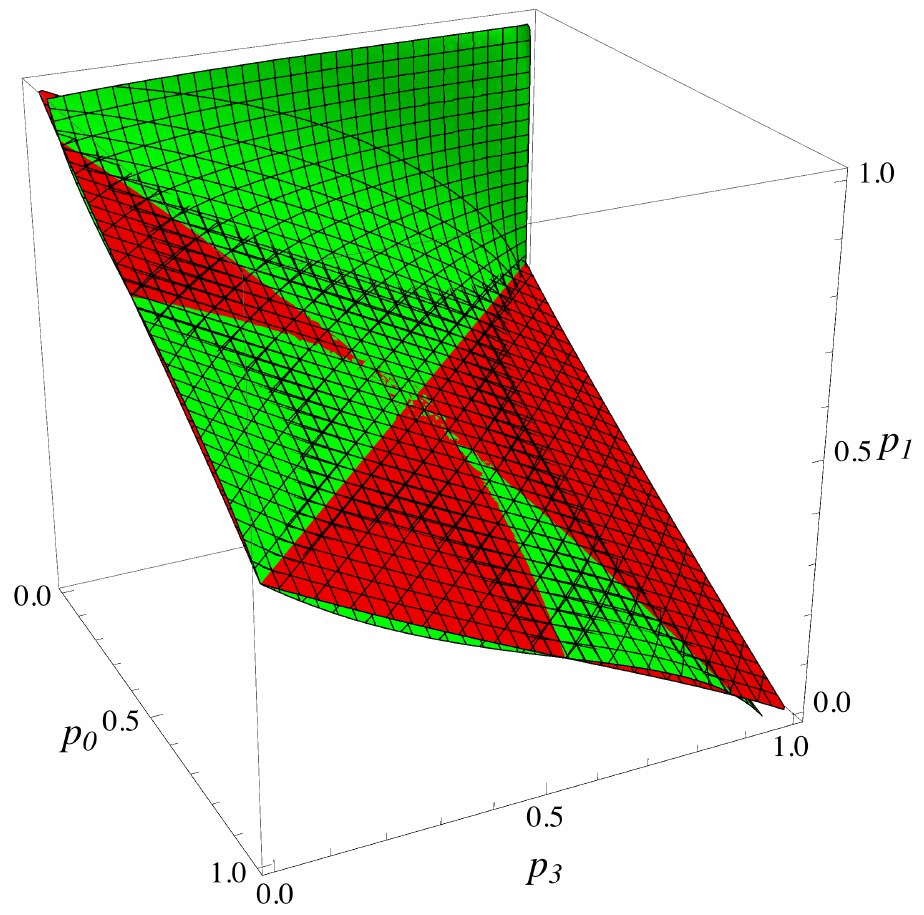}
\includegraphics[width = 2.2in]{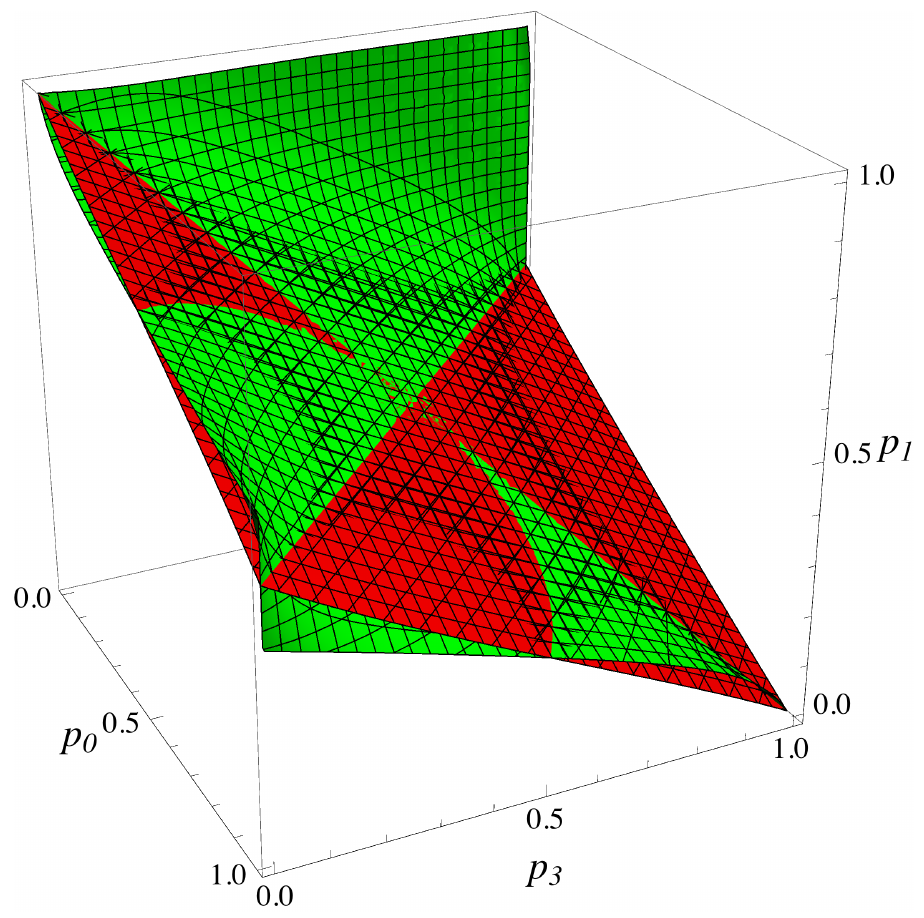}
\caption{\label{fig}When $N=3$ (upper left), $N=4$ (upper right), $N=5$ (lower left), or $N=6$ (lower right), the green (or light) surface is the surface $\mu_B=0$, and the red (or dark) surface is the surface $\mu_C=0$, both in the $(p_0,p_3,p_1)$ unit cube.  The Parrondo region is the region on or below the green surface and above the red surface, while the anti-Parrondo region is the region on or above the green surface and below the red surface.  (Assumptions: $p=1/2$ in game $A$, $p_1=p_2$ in games $B$ and $C$, and $\gamma=1/2$ in game $C$.)}
\end{figure}
\afterpage{\clearpage}

Toral's \cite{T} games were inspired by the one-player history-dependent games introduced by Parrondo, Harmer, and Abbott \cite{PHA}; in their game $B$, the player tosses a $p_0$-coin, a $p_1$-coin, a $p_2$-coin, or a $p_3$-coin if his two previous results are loss-loss, loss-win, win-loss, or win-win, respectively, with the second result being the more recent one.  There is an interesting relationship between Toral's three-player game $B$ and the one-player history-dependent game $B$.

For $i=1,2,3$, let ${\bm P}_i$ denote the $8\times8$ one-step transition matrix for the Markov chain in $\Sigma$ with $N=3$ corresponding to player $i$ being chosen to play at each turn (as usual, the coin tossed depends on the status of the nearest neighbors).  Then the one-step transition matrix ${\bm P}$ for Toral's game $B$ with $N=3$ can be defined as
${\bm P} := (1/3)({\bm P}_1 + {\bm P}_2 + {\bm P}_3)$.  It can be shown that the stationary distribution of ${\bm P}$ coincides with the stationary distribution of ${\bm P}_1{\bm P}_2{\bm P}_3$ when $p_1=p_2$.  The one-step transition matrix ${\bm P}_1{\bm P}_2{\bm P}_3$ corresponds to  the game in which the three players 1, 2, and 3 toss coins in the order named, and this game can be coupled with the one-player history-dependent game, using the same sequence of coin tosses, because the two most recent players are also the two nearest neighbors.

The result is that the mean profit $\mu_B$ in (\ref{meanB_N=3,p1=p2}) is the same as the mean profit per turn for the one-player history-dependent Parrondo game $B$, assuming $p_1=p_2$.  This then implies that the Parrondo regions are also the same, that is, the Parrondo region for the one-player history-dependent games is given by (\ref{Parrondo_N=3}).

\subsection{$N=4$}

Again we have an explicit formula for $\mu_B$ (algebraically equivalent to that of Xie et al.\ \cite{Xetal}), from which we can derive a similar formula for $\mu_C$.  Specifically, $\mu_B=\mu(p_0,p_1,p_1,p_3)$, where
$$
\mu(p_0,p_1,p_1,p_3):={f_0(p_0,p_3) +
 4 (1 + p_0) (q_0 + p_3) q_3 p_1 - 2 (q_0 + p_3) (q_0 - p_3) p_1^2\over g_0(p_0,p_1,p_3)}
$$
with $f_0(p_0,p_3):=-(3 - 2 p_3 - 3 p_0^2 + 2 p_0 p_3 - p_3^2 + 2 p_0^2 p_3 - 2 p_0 p_3^2)$ and $g_0(p_0,p_1,p_3):=(3 + 6 p_0 - 2 p_3 - 3 p_0^2 - 2 p_0 p_3 - p_3^2 + 12 p_0^2 p_3 - 4 p_0 p_3^2 - 8 p_0^2 p_3^2) - 4 (q_0 + p_3 + 2 p_0^2 + 2 p_0 p_3) q_3 p_1 + 2 (1 + 4 p_0 - p_0^2 - 2 p_0 p_3 - p_3^2) p_1^2$, and $\mu_C=\mu((1/2+p_0)/2,(1/2+p_1)/2,(1/2+p_1)/2,(1/2+p_3)/2)$.
The condition for the Parrondo effect amounts to two quadratic inequalities in $p_1$ with polynomial coefficients in $p_0$ and $p_3$.  Solving these inequalities and assuming $0<p_m<1$ for $m=0,1,3$, we find that the Parrondo region is described by $p_0+p_3<1$,
$$
p_1\le{(1 + p_0)  q_3 - \sqrt{ (1 + p_0)^2 q_3^2 + (q_0 - p_3)f(p_0,p_3)}\over q_0 - p_3}
$$
if $(1 + p_0)^2 q_3^2 + (q_0 - p_3)f(p_0,p_3)\ge0$,
where $f(p_0,p_3):=[p_0 (3 p_0 - 2 p_3 - 2 p_0 p_3 + 2 p_3^2) - (3+p_3)q_3]/[2 (q_0 + p_3)]$, and
$$
p_1> {g(p_0,p_3) - \sqrt{g(p_0,p_3)^2 + 4 (q_0 - p_3) h(p_0,p_3)}\over 4 (q_0 - p_3)},
$$
where $g(p_0,p_3):=13 + 8 p_0 - 8 p_3 - 4 p_0 p_3$ and
$h(p_0,p_3):=(-48 + 14 p_0 + 30 p_3 + 13 p_0^2 - 8 p_0 p_3 + 3 p_3^2 - 4 p_0^2 p_3 + 4 p_0 p_3^2)/(1 + q_0 + p_3)$.  We note that $p_1>1/2$ in the Parrondo region.  See Figure~\ref{fig}.  The region is connected, and a numerical integration yields 0.0293350 as its approximate volume.

\subsection{$N=5,6$}

When $N=5$ or $N=6$, we have explicit, albeit very complicated, formulas for $\mu_B$ from which we can derive similar formulas for $\mu_C$; in the case $N=6$, our formula assumes that $p_1=p_2$ for simplicity.  It would be impractical and uninformative to state those formulas here.  Instead, see Figure~\ref{fig}, which shows that the Parrondo regions in the two cases are surprisingly similar.  In both cases they appear to be the union of three connected components.

There is a subtle technical issue here.  We can solve the quadratic inequalities in the $N=4$ case, but the quartic and octic inequalities in the $N=5$ and $N=6$ cases are less tractable.  We expect that $\mu_B<0$ below the surface $\mu_B=0$ and $\mu_C>0$ above the surface $\mu_C=0$, but we do not have a proof.

\subsection{$7\le N\le 19$}

When $N\ge7$, we no longer have explicit formulas for $\mu_B$ but we can in principle compute it exactly for arbitrary values of the probability parameters by enumerating the state space $\Sigma/$$\sim$ (i.e., the set of all equivalence classes) and calculating the one-step transition matrix $\bar{\bm P}$ and the related column vector $\dot{\bar{\bm P}}\bm1$ as functions of $p_m$ and $q_m$ (as if $p_m$ and $q_m$ were unrelated) for $m=0,1,2,3$.  We then specify the desired numerical values of the probability parameters $p_m$ (and set $q_m:=1-p_m$) and evaluate the unique stationary distribution $\bar{\bm\pi}$.  Finally, we use the rightmost expression in (\ref{rule}) for $\mu_B$.  The advantage of this approach is that, once $\bar{\bm P}$ and $\dot{\bar{\bm P}}\bm1$ are found, we can apply (\ref{rule}) for arbitrary choices of the probability parameters without the time-consuming re-enumeration of the equivalence classes.

Computations were done on a MacBook Air with 2 GB of RAM using \textit{Mathematica 8}.  The \textit{Mathematica} program we used is displayed in the Appendix for $N=10$; it is also available at \url{http://www.math.utah.edu/~ethier/program.txt} or \url{http://yu.ac.kr/~leejy/program.txt}.  For $N\ge17$, we needed more memory.  Computations were done on an IBM System x3850 X5 with 1 TB of RAM using a Linux version of \textit{Mathematica 8}.  The runtime for the case $N=20$ was estimated at over seven weeks, so we stopped with $N=19$.

\begin{figure}[ht]
\centering
\includegraphics[width = 2.2in]{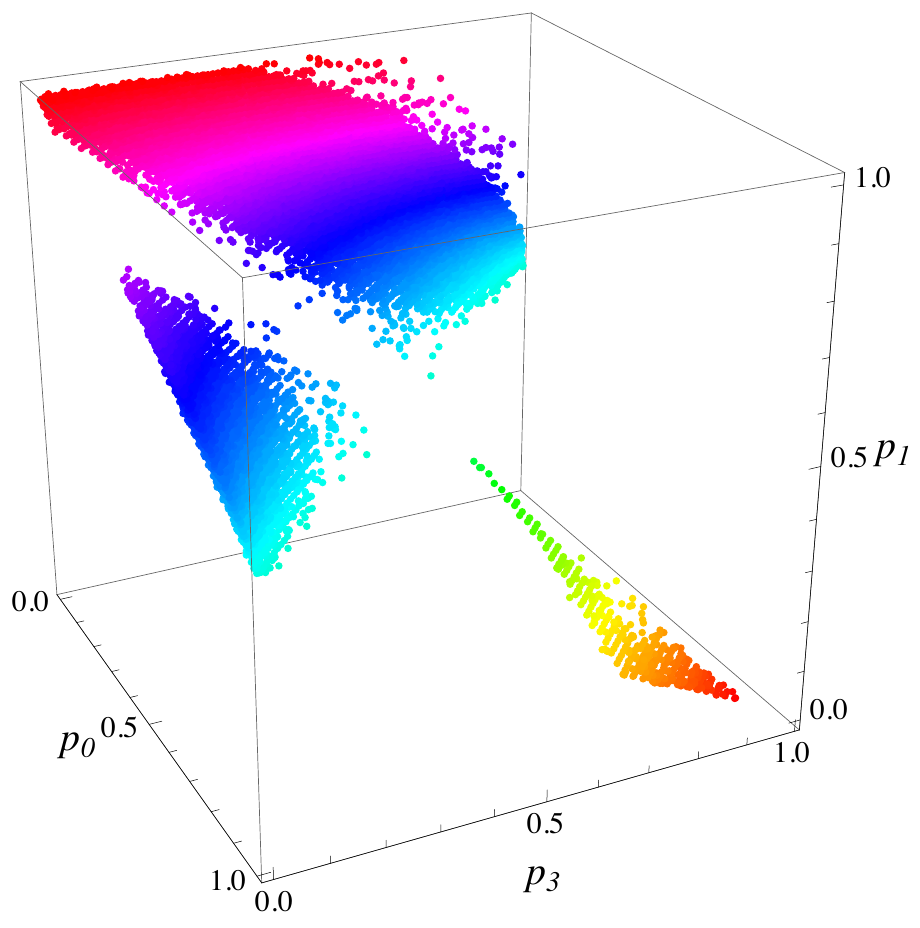}
\includegraphics[width = 2.2in]{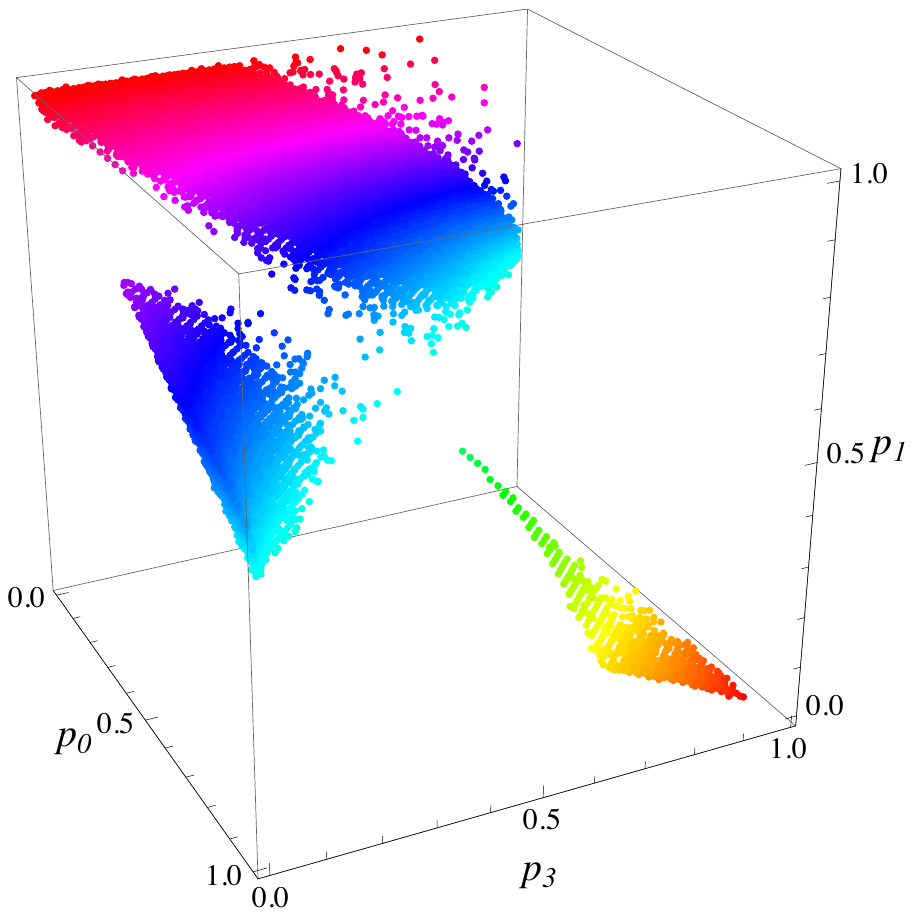}
\includegraphics[width = 0.32in]{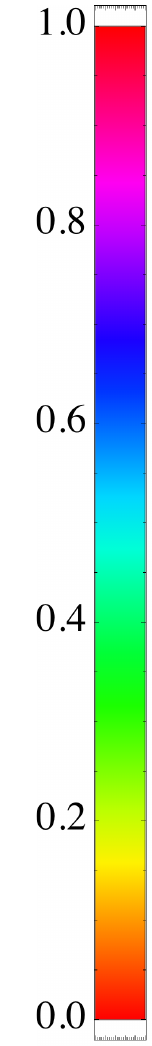}
\caption{\label{fig2}When $N=7$ (left) or $N=8$ (right), the Parrondo region is sketched by plotting points that belong to it.  Two million points were tested for each figure, one million equally spaced points and one million randomly chosen points.  Because this figure lacks the three-dimensionality of Figure~\ref{fig}, we use color (or shading) as a partial substitute.  The point at $(p_0,p_3,p_1)$ is colored using the \textit{Mathematica} function Hue[$p_1$].  As $h$ varies from 0 to 1, the color corresponding to Hue[$h$] runs through red, yellow, green, cyan, blue, magenta, and back to red again.  See the color scale on the right.  (Assumptions: See the caption to Figure~\ref{fig}.)}
\end{figure}

\subsection{Comparison}

Table~\ref{volume} summarizes our estimates of the volume of the Parrondo region.  For $3\le N\le13$ we used two methods for estimating this volume.  The first, which might be called the Riemann sum approximation, consists of evaluating $\mu_B$ and $\mu_C$ at each of the points $(2i+1,2j+1,2k+1)/200$ for $i,j,k=0,1,\ldots,99$, and determining the proportion of such points at which the Parrondo effect is present.  (This is just the Riemann sum for the indicator function of the Parrondo region.)  The second method is by simulation.  Our simulation estimate is also based on $10^6$ points but they are chosen randomly from the uniform distribution over the unit cube.  The Riemann sum approximation has the advantage of being repeatable, whereas the simulation estimate allows estimation of the error in the approximation.

\begin{table}[ht]
\caption{\label{volume}Estimated volume of the Parrondo region.  The integrated volume is obtained by integration or numerical integration and is rounded to six significant digits.  For the Riemann sum approximation, the unit cube is divided into $(100)^3$ cubes of equal size, then the estimate is the proportion of cubes for which the Parrondo effect is present at the center point.  The simulation estimate $\hat p$ is also based on $n=10^6$ points but here they are randomly chosen.  The standard error is the square root of $\hat{p}(1-\hat{p})/n$.    (Assumptions: See the caption to Figure~\ref{fig}.)\medskip}
\catcode`@=\active \def@{\hphantom{0}}
\catcode`#=\active \def#{\hphantom{$-$}}
\begin{center}
\begin{footnotesize}
\begin{tabular}{ccccc}
\hline
\noalign{\smallskip}
 $N$ & integrated &  Riemann sum  & simulation & standard \\
     &  volume    & approximation & estimate   &  error of   \\
     &            & to volume     & of volume  & simulation  \\
\noalign{\smallskip}
\hline
\noalign{\smallskip}
@3 & 0.0174361 & 0.017314 & 0.017336   &  0.0001305    \\
@4 & 0.0293350 & 0.029199 & 0.029227   &  0.0001684    \\
@5 &           & 0.011275 & 0.011521   &  0.0001067    \\
@6 &           & 0.010751 & 0.011090   &  0.0001047    \\
@7 &           & 0.008327 & 0.008671   &  0.0000927    \\
@8 &           & 0.007781 & 0.008028   &  0.0000892    \\
@9 &           & 0.007060 & 0.007372   &  0.0000855    \\
10 &           & 0.006776 & 0.006952   &  0.0000831    \\
11 &           & 0.006491 & 0.006791   &  0.0000821    \\
12 &           & 0.006356 & 0.006622   &  0.0000811    \\
13 &           & 0.006227 & 0.006492   &  0.0000803    \\
\noalign{\smallskip}
\hline
\end{tabular}
\end{footnotesize}
\end{center}
\end{table}

The volume of the Parrondo region appears to be getting smaller as $N$ increases.  We expect that it will converge to a nonzero limit, based not on Table~\ref{volume} but on Tables \ref{Toral} and \ref{two points} as we now explain.

\MR\ \cite{MR} studied the presence of the Parrondo effect in the case of Toral's choice of the probability parameters, namely $p=1/2$, $p_0=1$, $p_1=p_2=4/25$, and $p_3=7/10$ (actually, they took $p=0.499$).  They obtained estimates of $\mu_B$ and $\mu_C$ using exact computations for $N=3$, unspecified analytic methods for $4\le N\le 12$, and simulation for $N>12$.  Table~\ref{Toral} confirms their findings using exact computations of $\mu_B$ and $\mu_C$.  For example, $\mu_B=-599823882743/31695346763173$ when $N=6$, which we round to six significant digits in Table~\ref{Toral}.  Notice that $\mu_C$ seems to have stabilized to six significant digits by $N=11$, so it appears that \textit{all} of the variation in $\mu_C$ in the upper graph of Figure~2 of \cite{MR} is due to experimental error.  However, $\mu_B$ is more oscillatory.  It seems to have stabilized to three significant digits by $N=19$, so it appears that \textit{most} of the variation in $\mu_B$ in the lower graph of the same figure is due to experimental error.  This explains why we have not used simulation to estimate $\mu_B$ and $\mu_C$ for $N\ge20$.

Table~\ref{two points} analyzes two other cases, a second point on the boundary of the unit cube and a point in the interior.  In both cases $\mu_B$ seems to have stabilized more quickly than in the case of Toral's choice of the probability parameters.

\begin{table}[ht]
\caption{\label{Toral}Analysis of the Parrondo effect for Toral's choice of the probability parameters.  Ellipses are intended to suggest that exact numbers have been truncated at six digits.  The last two columns are rounded to six significant digits.  (Assumptions: See the caption to Figure~\ref{fig}.)\medskip}
\catcode`@=\active \def@{\hphantom{0}}
\catcode`#=\active \def#{\hphantom{$-$}}
\begin{center}
\begin{footnotesize}
\begin{tabular}{cccc}
\hline
\noalign{\smallskip}
 $N$ & Parrondo $p_1$-interval & $\qquad\mu_B$@@ & $\qquad\mu_C$@@ \\
     & when $(p_0,p_3)=(1,7/10)$ & \multicolumn{2}{c}{at $(p_0,p_1,p_3)=(1,4/25,7/10)$} \\
\noalign{\smallskip}
\hline
\noalign{\smallskip}
@3 & $(0.195651\cdots,0.230769\cdots]$ & $-0.0909091$@@ & $-0.0183774$@ \\
@4 & --- empty ---                     & #0.0799608@@   & #0.0171357@ \\
@5 & $(0.150762\cdots,0.162596\cdots]$ & $-0.00219465$@ & #0.00405176 \\
@6 & $(0.149365\cdots,0.178102\cdots]$ & $-0.0189247$@@ & #0.00463310 \\
@7 & $(0.148884\cdots,0.155594\cdots]$ & #0.00350598@   & #0.00482261 \\
@8 & $(0.148968\cdots,0.159157\cdots]$ & #0.000698188   & #0.00479021 \\
@9 & $(0.148967\cdots,0.162158\cdots]$ & $-0.00189233$@ & #0.00479036 \\
10 & $(0.148966\cdots,0.160394\cdots]$ & $-0.000332809$ & #0.00479099 \\
11 & $(0.148966\cdots,0.160550\cdots]$ & $-0.000466527$ & #0.00479089 \\
12 & $(0.148966\cdots,0.160793\cdots]$ & $-0.000676916$ & #0.00479089 \\
13 & $(0.148966\cdots,0.160662\cdots]$ & $-0.000562901$ & #0.00479089 \\
14 & $(0.148966\cdots,0.160669\cdots]$ & $-0.000569340$ & #0.00479089 \\
15 & $(0.148966\cdots,0.160689\cdots]$ & $-0.000586184$ & #0.00479089 \\
16 & $(0.148966\cdots,0.160680\cdots]$ & $-0.000578161$ & #0.00479089 \\
17 & $(0.148966\cdots,0.160680\cdots]$ & $-0.000578345$ & #0.00479089 \\
18 & $(0.148966\cdots,0.160681\cdots]$ & $-0.000579652$ & #0.00479089 \\
19 & $(0.148966\cdots,0.160681\cdots]$ & $-0.000579095$ & #0.00479089 \\
\noalign{\smallskip}
\hline
\end{tabular}
\end{footnotesize}
\end{center}
\end{table}

\begin{table}[ht]
\caption{\label{two points}Analysis of the Parrondo effect for a second point on the boundary of the unit cube and for a point in the interior.  (Assumptions: See the caption to Figure~\ref{fig}.)\medskip}
\catcode`@=\active \def@{\hphantom{0}}
\catcode`#=\active \def#{\hphantom{$-$}}
\begin{center}
\begin{footnotesize}
\begin{tabular}{cccc}
\hline
\noalign{\smallskip}
 $N$ & Parrondo $p_1$-interval & $\qquad\mu_B$@@ & $\qquad\mu_C$@@ \\
     & when $(p_0,p_3)=(7/10,0)$ & \multicolumn{2}{c}{at $(p_0,p_1,p_3)=(7/10,17/25,0)$} \\
\noalign{\smallskip}
\hline
\noalign{\smallskip}
@3 & --- empty ---                     & #0.0710383@     & #0.0297791@ \\
@4 & $(0.672790\cdots,0.807540\cdots]$ & $-0.0425713$@   & #0.00241457 \\
@5 & $(0.657367\cdots,0.675341\cdots]$ & #0.00257895     & #0.00818232 \\
@6 & $(0.659797\cdots,0.699307\cdots]$ & $-0.0102930$@   & #0.00721881 \\
@7 & $(0.659410\cdots,0.694010\cdots]$ & $-0.00722622$   & #0.00736816 \\
@8 & $(0.659472\cdots,0.695419\cdots]$ & $-0.00808338$   & #0.00734464 \\
@9 & $(0.659462\cdots,0.695052\cdots]$ & $-0.00784318$   & #0.00734835 \\
10 & $(0.659463\cdots,0.695147\cdots]$ & $-0.00790952$   & #0.00734776 \\
11 & $(0.659463\cdots,0.695122\cdots]$ & $-0.00789119$   & #0.00734786 \\
12 & $(0.659463\cdots,0.695129\cdots]$ & $-0.00789624$   & #0.00734784 \\
13 & $(0.659463\cdots,0.695127\cdots]$ & $-0.00789485$   & #0.00734784 \\
14 & $(0.659463\cdots,0.695128\cdots]$ & $-0.00789523$   & #0.00734784 \\
15 & $(0.659463\cdots,0.695127\cdots]$ & $-0.00789513$   & #0.00734784 \\
16 & $(0.659463\cdots,0.695127\cdots]$ & $-0.00789516$   & #0.00734784 \\
17 & $(0.659463\cdots,0.695127\cdots]$ & $-0.00789515$   & #0.00734784 \\
18 & $(0.659463\cdots,0.695127\cdots]$ & $-0.00789515$   & #0.00734784 \\
19 & $(0.659463\cdots,0.695127\cdots]$ & $-0.00789515$   & #0.00734784 \\
\noalign{\smallskip}
\hline
\noalign{\smallskip}
 $N$ & Parrondo $p_1$-interval & $\qquad\mu_B$@@ & $\qquad\mu_C$@@ \\
     & when $(p_0,p_3)=(1/10,3/4)$ & \multicolumn{2}{c}{at $(p_0,p_1,p_3)=(1/10,3/5,3/4)$} \\
\noalign{\smallskip}
\hline
\noalign{\smallskip}
@3 & $(0.611111\cdots,0.714285\cdots]$ & $-0.190476$@@  & $-0.00671141$ \\ 
@4 & $(0.584416\cdots,0.640975\cdots]$ & $-0.0858189$@  & #0.0108365@ \\ 
@5 & $(0.580262\cdots,0.616548\cdots]$ & $-0.0389980$@  & #0.0141217@ \\ 
@6 & $(0.579542\cdots,0.607387\cdots]$ & $-0.0183165$@  & #0.0147166@ \\ 
@7 & $(0.579415\cdots,0.603644\cdots]$ & $-0.00924232$  & #0.0148223@ \\ 
@8 & $(0.579393\cdots,0.602063\cdots]$ & $-0.00528548$  & #0.0148408@ \\ 
@9 & $(0.579390\cdots,0.601387\cdots]$ & $-0.00356984$  & #0.0148441@ \\ 
10 & $(0.579389\cdots,0.601097\cdots]$ & $-0.00282963$  & #0.0148446@ \\ 
11 & $(0.579389\cdots,0.600973\cdots]$ & $-0.00251155$  & #0.0148447@ \\ 
12 & $(0.579389\cdots,0.600920\cdots]$ & $-0.00237531$  & #0.0148447@ \\ 
13 & $(0.579389\cdots,0.600897\cdots]$ & $-0.00231709$  & #0.0148448@ \\ 
14 & $(0.579389\cdots,0.600887\cdots]$ & $-0.00229226$  & #0.0148448@ \\ 
15 & $(0.579389\cdots,0.600883\cdots]$ & $-0.00228169$  & #0.0148448@ \\ 
16 & $(0.579389\cdots,0.600881\cdots]$ & $-0.00227719$  & #0.0148448@ \\ 
17 & $(0.579389\cdots,0.600881\cdots]$ & $-0.00227528$  & #0.0148448@ \\ 
18 & $(0.579389\cdots,0.600880\cdots]$ & $-0.00227446$  & #0.0148448@ \\ 
19 & $(0.579389\cdots,0.600880\cdots]$ & $-0.00227412$  & #0.0148448@ \\ 
\noalign{\smallskip}
\hline
\end{tabular}
\end{footnotesize}
\end{center}
\end{table}
\afterpage{\clearpage}

\section{Conclusions}

We considered the spatially dependent Parrondo games of Toral \cite{T}, which assume $N\ge3$ players arranged in a circle, and in which the win probability for a player depends on the status of the player's two nearest neighbors.  There are three games, game $A$ without spatial dependence, game $B$ with spatial dependence, and the randomly mixed game $C:=(1/2)(A+B)$.  The model is described by a parameterized Markov chain with $2^N$ states.  To maximize the value of $N$ for which exact computations are feasible, we regarded states as equivalent if they are equal after a rotation and/or reflection of the players.  This allowed us to compute the mean profits per turn, $\mu_B$  and $\mu_C$, to the ensemble of $N$ players for $3\le N\le19$ and several choices of the parameter vector $(p_0,p_1,p_3)$, including Toral's choice.  The results provide numerical evidence, but not a proof, that $\mu_B$ and $\mu_C$ converge as $N\to\infty$ and that the Parrondo effect (i.e., $\mu_B\le0$ and $\mu_C>0$) persists for all $N$ sufficiently large for a set of parameter vectors having nonzero volume.  This suggests that the spatially-dependent version of Parrondo's paradox is a robust phenomenon that remains present in the thermodynamic limit.

This is the main conclusion, but there are several other noteworthy conclusions.  We have shown that the sequence of profits to the ensemble of $N$ players obeys the strong law of large numbers.  This is important in defining what is meant by a winning, losing, and fair game.  For $3\le N\le 6$ explicit formulas for $\mu_B$ and $\mu_C$ are available, so with the help of computer graphics, we have demonstrated that one can visualize the Parrondo region, the region in the three-dimensional parameter space in which the Parrondo effect appears.  There is also an anti-Parrondo region, and we have shown that it is symmetric with the Parrondo region, as might be expected.  Finally, we have pointed out a close relationship between the spatially dependent Parrondo games in the case of $N=3$ players and the one-player history-dependent Parrondo games.

We have restricted our attention to the one-dimensional version of the model, but our methods may have applicability to the two-dimensional version, already investigated by Mihailovi\'c and Rajkovi\'c \cite{MR3} using computer simulation.  There are also, of course, several problems for mathematicians, including a proof that $\mu_B$ and $\mu_C$ actually do converge as $N\to\infty$.  We have partial results in this direction, which relate the Markov chain model to an interacting particle system, or spin system (Liggett \cite{L}, Chapter 3).  Spin systems are continuous-time Markov processes in the infinite-product space $\{0,1\}^S$, where $S$ is a countable set such as the $d$-dimensional integer lattice.  Some of the best-known spin systems (e.g., the stochastic Ising model) were motivated by physics.  It is interesting that two mathematical models (Parrondo games and spin systems), with rather different physical motivations, should come together in this way.

\section*{Acknowledgments}

We thank an anonymous referee for conjecturing the symmetry between the Parrondo and anti-Parrondo regions, now confirmed in Theorem~2.

The work of S. N. Ethier was partially supported by a grant from the Simons Foundation (209632).  It was also supported by a Korean Federation of Science and Technology Societies grant funded by the Korean Government (MEST, Basic Research Promotion Fund).  The work of J. Lee was supported by a Yeungnam University research grant from 2009.

\section*{Appendix}
\begin{figure}[ht]
\centering
\includegraphics[width = 4.8in]{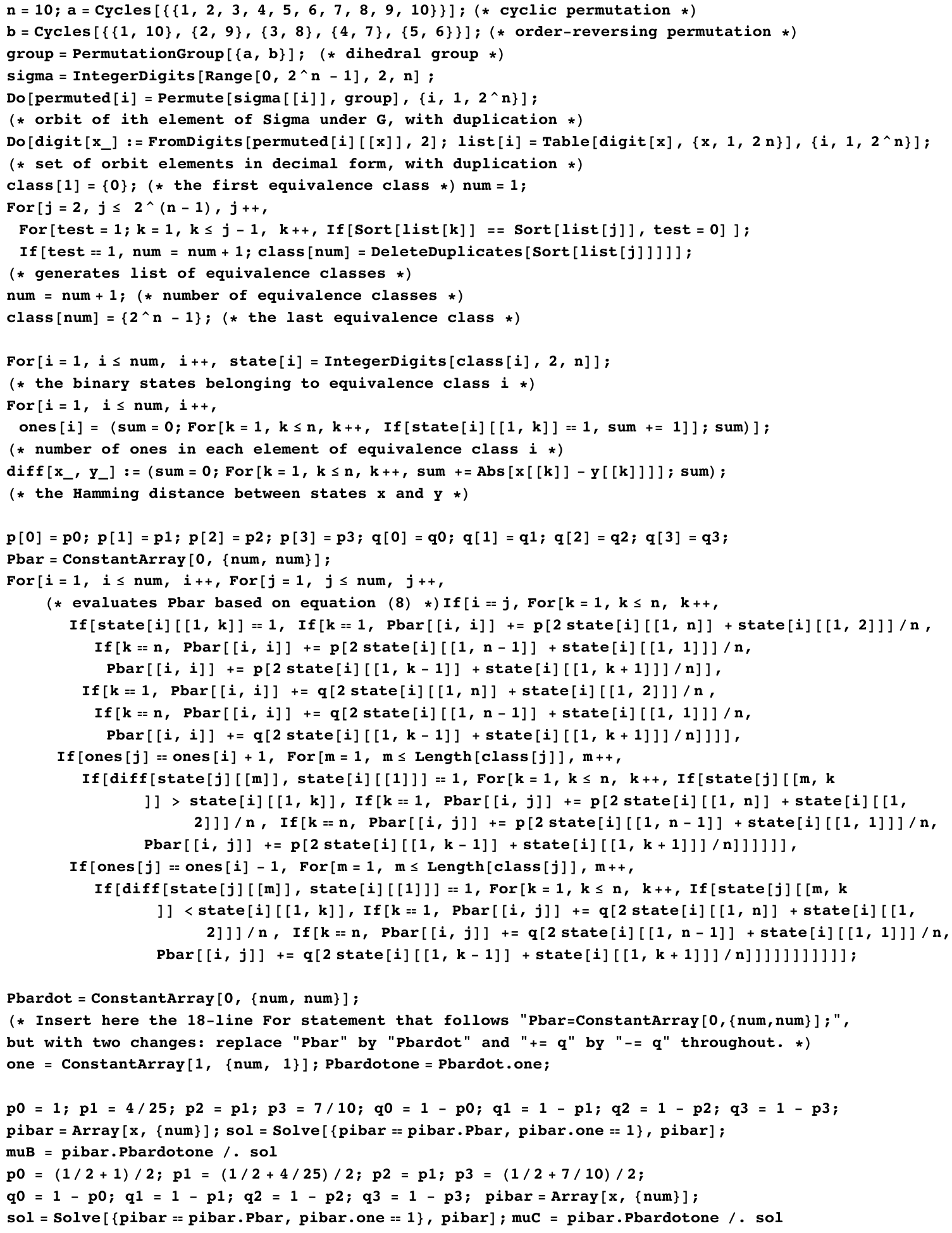}
\caption{\label{fig3}The \textit{Mathematica} program used to enumerate the equivalence classes, compute $\bar{\bm P}$ and $\dot{\bar{\bm P}}\bm1$, solve for the stationary distribution $\bar{\bm\pi}$, and evaluate $\mu_B$ and $\mu_C$.  Here $N=10$ and $p_1=p_2$.  But $N$ can be changed by modifying the first three instructions, and the assumption $p_1=p_2$ can be dropped by eliminating \texttt{b} in line 3 and changing \texttt{2n} to \texttt{n} in line 7.}
\end{figure}

\end{document}